\documentclass[11pt]{amsart}

\newcommand{\Zz}{\mathbb{Z}}
\newcommand{\Cc}{\mathbb{C}}
\newcommand{\Pp}{\mathbb{P}}

\newcommand{\la}{\langle}
\newcommand{\ra}{\rangle}

\newcommand{\Oo}{\mathcal{O}}

\def\mod{\mathop{\mathrm{mod}}\nolimits}

\def\Pic{\mathop{\mathrm{Pic}}}

\def\Proj{\mathop{\mathrm{Proj}}}

\def\Tr{\mathop{\mathrm{Tr}}\nolimits}

\makeatletter
\newtheorem*{rep@theorem}{\rep@title}
\newcommand{\newreptheorem}[2]{%
\newenvironment{rep#1}[1]{%
 \def\rep@title{#2 \ref{##1}}%
 \begin{rep@theorem}}%
 {\end{rep@theorem}}}
\makeatother

\newreptheorem{Thm}{Theorem}

\newreptheorem{Cor}{Corollary}

\newreptheorem{Con}{Conjecture}

\newtheorem{thm-int}{Theorem}

\theoremstyle{definition}
\newtheorem{Def-s}[Thm]{Definition}

\newtheorem{theorem}{Theorem}[section]

\newtheorem{proposition}[theorem]{Proposition}
\newtheorem{definition}[theorem]{Definition}

\newtheorem{corollary}[theorem]{Corollary}
\newtheorem{remark}[theorem]{Remark}

\numberwithin{equation}{section}

\begin{document}

\title{On mirrors of elliptically fibered K3 surfaces}

\begin{abstract}
We study a two-parameter family of K3 surfaces of (generic) Picard rank $18$
which is mirror to the $18$-dimensional family of elliptically fibered K3 surfaces with a section. Members of this family are given as compactifications
of hypersurfaces in three dimensional algebraic torus given by equations $y^2+z+z^{-1} + x^3 + ax +b =0$
where $a$ and $b$ are the parameters of the family. It follows from previous work of Morrison that these surfaces are double covers of Kummer surfaces coming from products of elliptic curves, and we establish this connection explicitly. This in turn provides a description of the one-parameter families of K3 surfaces which are mirror to polarized K3 surfaces of Picard rank one. 

{\bf It has been brought to our attention that the results of this paper have already appeared in the literature. The appropriate references are added at the end of the introduction.}
\end{abstract}


\author{Lev A. Borisov}
\address{Department of Mathematics\\
Rutgers University\\
Piscataway, NJ 08854} \email{borisov@math.rutgers.edu}

\maketitle

\section{Introduction}
Mirror symmetry for K3 surfaces got its start in the work of Dolgachev and Nikulin (and independently Pinkham). The most authoritative reference in the paper of Dolgachev \cite{Dolgachev}. In its most superficial description,
it associates to some $k$-dimensional families of algebraic K3 surfaces of generic Picard rank $20-k$ some $(20-k)$-dimensional families of $K3$ surfaces 
of generic Picard rank $k$. 

\medskip
More precisely, suppose that $M$ is primitive sublattice of signature $(1,\cdot)$ of  the second integer cohomology lattice of a K3 surface
$$
M\subset H^2({\rm K3},\Zz) = E_8(-1)\oplus E_8(-1) \oplus U\oplus U\oplus U
$$
where $E_8$ is the unique positive definite even unimodular lattice of rank eight and $U$ is the unimodular lattice of signature $(1,1)$ given 
by the pairing matrix
$$
\left(
\begin{array}{cc}
0&1\\1&0
\end{array}
\right)
.
$$
In its simplest form, Dolgachev's mirror to the family of $M$-polarized K3 surfaces is a family of $M^\vee$-polarized K3 surfaces where 
$M^\vee$ is an orthogonal complement to some copy of $U$ in the orthogonal complement $M^\perp$ of $M$ in $ H^2({\rm K3},\Zz)$.

\medskip
The set of algebraic K3 surfaces is a countable union of the $19$-dimensional families marked with $\la 2n\ra$, i.e. K3 surfaces with a choice of an ample (or more generally semi-ample) divisor class $D$ with $D^2 =2n$. Naturally, these are some of the most studied families of K3 surfaces, see \cite{GNS}. So it makes perfect sense to try to better understand their mirrors.
It was shown in \cite{Dolgachev} that the Picard lattices of the (very general) members of these mirror families are isomorphic to
$$
E_8(-1)\oplus E_8(-1) \oplus U \oplus \la -2n \ra
$$
and their moduli spaces are isomorphic to the well-known moduli curves $X_0(n)^+$.
The elements of these families are shown to be Shioda-Inose partners of the products of elliptic curves related by a Fricke involution. Specifically, a generic member of this family 
is birational to a double cover of the Kummer surface $E_1\times E_2$ where $E_1$ and $E_2$ are isogeneous elliptic curves that correspond to the parameters 
$\tau$ and $(-\frac 1{n\tau})$.


\medskip
For each $n$ one can extend a primitive element of $H^2({\rm K3},\Zz)$ of square $2n$ to a sublattice isomorphic to $U$.
It is therefore natural to consider a larger two-dimensional family which is the Dolgachev mirror to the family of $U$-marked K3
surfaces.  The Picard group of a generic element of the mirror is marked with $E_8(-1)\oplus E_8(-1) \oplus U$. These surfaces are Shioda-Inose partners 
of the products of two elliptic curves $E_1\times E_2$ for a generic choice of the unordered pair of elliptic curves.

\medskip
This paper gives a very explicit description of these Shioda-Inose partners. The main result is the following theorem.

\medskip
\medskip
{\bf Theorem \ref{main}.}
For generic choices of elliptic curves $E_1$ and $E_2$ with $J$-invariants $j_1$ and $j_2$ the 
surface $X$ which is the compactification of
the solution space of
$$
y^2 + z+z^{-1} + x^3  -\frac{j_1^{\frac 13}j_2^{\frac 13} }{48}\,x  -\frac{ (j_1-1728)^{\frac 12}(j_2-1728)^{\frac 12}}{864} = 0
$$
and $E_1\times E_2$ form a Shioda-Inose pair. Specifically, the minimal resolution 
of the quotient of $X$ by $\mu:(x,y,z)\mapsto (x,-y,z^{-1})$ 
is isomorphic to the Kummer surface of $E_1\times E_2$.

\medskip
\medskip
In particular, when $E_1$ and $E_2$ are related by the Fricke involution $\tau\mapsto( -\frac 1{n\tau})$, the above equation describes the members of the Dolgachev's one-dimensional families (in $n=1$ case the surface has a node, so one needs to consider its resolution). This reduces the description of these surfaces to the classical problem of modular polynomials, see for example \cite{Silverman}.


\medskip
The paper is organized as follows. In Section \ref{toricsection} we use Batyrev's description of mirror symmetry for hypersurfaces  in toric varieties to construct
families of K3 surfaces $X$ with the desired lattice $E_8(-1)\oplus E_8(-1)\oplus U$. The main idea is that the $U$-marked K3 surfaces can be viewed as resolutions
of singularities of generic degree $12$ surfaces in the weighted projective space $W\Pp(1,1,4,6)$. We also describe a natural Morrison-Nikulin involution $\mu$ on $X$ and the resulting quotient K3 surfaces $\widetilde{X/\mu}$.  In Section \ref{secKummer} we recall the geometry of Kummer surfaces $Y$ associated to the product of elliptic curves, in order to fix our notations. In Section \ref{seciso} we describe an isomorphism between surfaces $Y$ and $\widetilde{X/\mu}$ by identifying generators in
their Picard lattices. However, this isomorphism is not explicit, in the sense that it does not provide the identification of the parameters. This is rectified in Section \ref{secexp}, where we prove out main result. Finally, in Section \ref{secdim1} we apply our construction to the one-parameter families of Dolgachev.

\medskip
{\bf Acknowledgements.} This project grew from informal discussions with Nick Sheridan, who organized a reading group on Homological Mirror Symmetry for K3 surfaces at the Institute for Advanced Study. I have been partially supported by the NSF Grants  DMS-1201466 and DMS-1601907, as well as by the Simons Foundation Fellowship. I also would like to thank the Institute for Advanced Study for its hospitality during the year-long program in Homlogical Mirror Symmetry.

\medskip
{\bf Addendum.} {\bf After the first version of the preprint has appeared, it was brought to our attention that the results are known to the experts.}
Specifically, the results appear in the work of Shioda \cite{ShiodaSandwich}, based on the work of Kuwata and Shioda \cite{KuwataShioda}. They also appear in 
the work of  Clingher and Doran \cite{DoranClingher}. I would like to thank Chuck Doran and Yuya Matsumoto for bringing this to my attention. I would like to also thank Igor Dolgachev for historical comments regarding mirror symmetry for K3 surfaces.

\section{Toric geometry of the mirror family.}\label{toricsection}
In this section we explain how to express an elliptically  K3 surface as an anticanonical hypersurface in the weighted projective space $W\Pp(1,1,4,6)$. We then describe its mirror and observe that it has the expected Picard marking coming from the ambient toric variety.

\medskip
Let $Z$ be a K3 surface with an elliptic fibration and a section. By this we mean a map
$$
\mu: Z\mapsto \Pp^1
$$
with genus one fibers and a section $S\subset Z$, which is a smooth rational curve with $S^2 =-2$. Together with the class $F$ of the fiber, the class of $S$ forms a sublattice $U\subseteq \Pic(Z)$ which is isomorphic to the standard hyperbolic lattice. In other words, we get a $U$-marking on $Z$. It can be shown that a $U$-marking in turn yields an elliptic fibration with section, see \cite{Huybrechts}.

\medskip
Generically, $\Pic(Z)=\Zz^2$ so the fibration $X_U\to \Pp^1$ has irreducible fibers, which we will assume from now on. 
Taking the quotient by the Kummer involution $(-)$ on each fiber gives a fibration 
$$
\pi:Y=Z/(-) \to \Pp^1
$$
whose fibers are isomorphic to $\Pp^1$. The image $S_1$ of $S$ satisfies
$$
S_1^2 = \frac 12 (\pi^*S_1)^2 = \frac 12 (2 S)^2 = -4,
$$
thus the quotient is the Hirzebruch surface $F_4$. 
Then the K3 surface $Z$ is the double cover of $F_4$ ramified at 
twice anticanonical class. In the other direction, 
note that on the Hirzebruch surface $F_4$ there is a fiber $F$ and the section $S_1$ with 
$S_1^2=-4,S_1F=1,F^2=0$. The canonical class is $(-2S-6F)$. The image of the ramification divisor lies in the linear system  $\vert 4S_1+12F\vert $. The linear system  system $\vert 4S_1+12F \vert $
has 
the base locus $S_1$, and generic member is given by $S_1$ plus a section of $3S_1+12F$ which is disjoint from $S_1$. Then 
the double cover is a smooth K3 surface $Z$ with the ramification locus of $Z\to F_4$ given by two components
that correspond to the zero section and the tri-section of nonzero order two points of the fibers. 

\medskip
Let us now realize these surfaces $Z$ as minimal resolutions of hypersurfaces in a toric variety, specifically in a weighted projective space.

\begin{proposition}
A K3 surface $Z$ with an elliptic fibration, a section, and irreducible fibers is the minimal resolution of singularities of 
a hypersurface of degree $12$ in the weighted projective space $W\Pp(1,1,4,6)$.
\end{proposition}

\begin{proof}
The Hirzebruch surface $F_4$ is the minimal resolution of singularities of the weighted projective plane $W\Pp(1,1,4)$
and $\Oo(3S_1+12F)$ is the pullback of the $\Oo(12)$ line bundle on $W\Pp(1,1,4)$. So we can think of the K3 surface $Z$ as 
the resolution of singularities of the double cover of $W\Pp(1,1,4)$ ramified at a divisor of degree $12$. This can be naturally viewed as a hypersurface in $W\Pp(1,1,4,6)$ given by a degree $12$ equation.
\end{proof}

\medskip
It has been suggested by Batyrev \cite{Bat.dual} that the mirror to this family of surfaces is obtained by switching 
the roles of two natural reflexive polytopes associated to a toric Fano threefold. This prescription was verified to coincide with Dolgachev's one 
in \cite{Rohsiepe}.

\medskip
Specifically, Batyrev's mirror symmetry produces a mirror family to K3 surfaces with elliptic fibration and a section whose general members are
compactifications $X$ of the hypersurface  $X_0\subset (\Cc^*)^3$ whose Newton polytope $\Delta$ is the simplex with vertices
\begin{equation}\label{vert}
v_1=(-1,-4,-6),~v_2=(1,0,0),~v_3=(0,1,0),~v_4=(0,0,1).
\end{equation}
Indeed, the relation $v_1 + v_2 + 4v_3 +6v_4=0$ comes from the weights of the  weighted projective space $W\Pp(1,1,4,6)$ and determines the above 
vertices uniquely up to isomorphism of $\Zz^3$.

\begin{remark}
If we denote the monomials that correspond to 
$$(0,1,1),~ (0,1,2),~ (-1,-2,-3),
$$
by $x$, $y$ and $z$ respectively, then up to an invertible monomial the equation of the hypersurface $X_0$ is
\begin{equation}\label{c-s}
0 = c_1 z + c_2 z^{-1} + c_3 + c_4 x + c_5 x^2 + c_6 x^3 + c_7 y + c_8 y^2 + c_9 xy
\end{equation}
with $z$, $z^{-1}$, $x^3$ and $y^2$ corresponding to the vertices of $\Delta$ in the order of \eqref{vert}.
By using $(\Cc^*)^3$ action on the coordinates $x,y,z$ as well as other symmetries of the toric variety, such as
$(x,y,z)\mapsto (\alpha_1x+\alpha_2,y+\alpha_3x + \alpha_4,z)$ one can reduce 
this equation to a very simple form
\begin{equation}\label{c-s-reduced}
0 = z+z^{-1} + y^2 + x^3 +ax+b.
\end{equation}
\end{remark}

This prompts the following definition.
\begin{definition}
We call the two parameter family of K3 surfaces which are compactifications $X=X(a,b)$ of the solution set 
of \eqref{c-s-reduced}  a mirror to the family of elliptically fibered K3 surfaces with a section. 
\end{definition}

In what follows, we will be studying these surfaces $X$ in great detail. We first need to realize them as hypersurfaces in 
nef-Fano toric varieties by resolving the ambient toric variety. This will allow us to find $19$ smooth rational curves inside each $X$ in a specific configuration.

\medskip
There is a natural compactification $X_{singular}$ of the set of solutions to \eqref{c-s} which is given by the closure inside the singular projective toric threefold $\Proj (\oplus_{k\geq 0} \Cc (k\Delta \cap \Zz^3))$. The fan of this 
toric threefold is given by the spans of the proper subsets of the set of vertices of  the simplex  $\Delta^\vee$ 
which is the convex hull of the four vertices 
\begin{equation}\label{list}
(-1,-1,-1),~(11,-1-1),~(-1,2,-1),~(-1,-1,1).
\end{equation}
The fan of the resolution of the ambient toric variety is given by a triangulation of the boundary of $\Delta^\vee$. 

\medskip
If we assume that $(a,b)$ are generic so that $X_{singular}$ is $\Delta$-regular in the sense of \cite{Bat.dual}, then the singularities of $X_{singular}$ are simply inherited from that of the ambient variety.
In order for
the closure of the proper preimage of $X_{singular}$ to be smooth, we only need to ensure that this triangulation involves all lattice points on the edges of $\Delta^\vee$. In other words, we allow the ambient variety to be singular at zero-dimensional torus orbits, since a $\Delta$-regular hypersurface is disjoint from these orbits. There are many such choices, and some of them make the ambient variety smooth, but they all give the same minimal resolution $X\to X_{singular}$.
 
\medskip
Let us now study the geometry of $X$.
Each facet of $\Delta$ gives a toric subvariety of $\Proj (\oplus_{k\geq 0} \Cc (k\Delta \cap \Zz^3))$ and we would like to understand the corresponding closed subsets in $X_{singular}$, as well as the singularities of $X_{singular}$ along the generic points of these strata.

\medskip 
The codimension one subvarieties of $X_{singular}$ are given as follows. We have  genus zero curves which are hypersurfaces in toric surfaces
$$
l_1 \subset \Proj (\oplus_{k\geq 0} \Cc (k\,{\rm Conv}(v_2,v_3,v_4) \cap \Zz^3)),
$$
and
$$
l_2 \subset \Proj (\oplus_{k\geq 0} \Cc (k\,{\rm Conv}(v_1,v_3,v_4) \cap \Zz^3))
$$
where the genus can be computed by looking at the number of interior lattice points of the Newton polygon \cite{Khovanskii}.
We also have 
$$
l_3 \subset \Proj (\oplus_{k\geq 0} \Cc (k\,{\rm Conv}(v_1,v_2,v_4) \cap \Zz^3))
$$
of genus one and 
$$
l_4 \subset \Proj (\oplus_{k\geq 0} \Cc (k\,{\rm Conv}(v_1,v_2,v_3) \cap \Zz^3))
$$
of genus two.

\medskip 
The intersections of the curves $l_i$ correspond to toric strata of dimension two in  $\Proj (\oplus_{k\geq 0} \Cc (k\Delta \cap \Zz^3))$. These give singular points on $X_{singular}$ precisely when the lattice length of the corresponding segment in $\Delta^\vee$ is larger than one. The number of points in the stratum is governed by 
the lattice length of the segment in $\Delta$. Specifically, we conclude that $X_{singular}$ 
has one singular point of type $A_{11}$ at $l_1\cap l_2$. It also has $A_2$ singularities
at $l_1\cap l_3$ and $l_2\cap l_3$. Finally, it has $A_{1}$ singularities at
$l_1\cap l_4$ and $l_2\cap l_4$. In all of these cases, the intersection consists of a single point. Note that 
$l_3\cap l_4$ consists of two points, but $X_{singular}$ is smooth there.

\medskip
Therefore, the minimal resolution $X$ of $X_{singular}$ has the following tree of smooth rational curves in it,
of self-intersection $(-2)$ each,  marked as $\bullet$ in the incidence graph below.
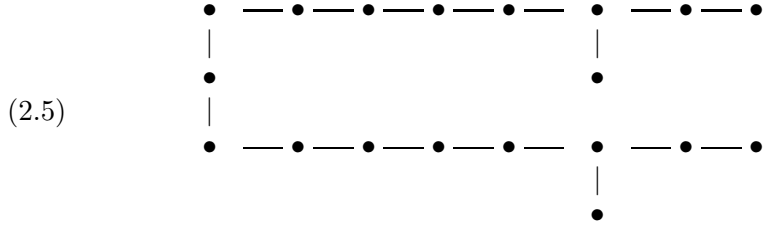
\begin{equation}\label{graph}
\begin{array} {cccl}
\bullet
& 
\stackrel{\line(1,0){15}}{}
\bullet 
\stackrel{\line(1,0){15}}{}
\bullet
 \stackrel{\line(1,0){15}}{}
 \bullet \stackrel{\line(1,0){15}}{}
 \bullet \stackrel{\line(1,0){15}}{}
&\bullet
& \stackrel{\line(1,0){15}}{}
 \bullet \stackrel{\line(1,0){15}}{}
 \bullet \\
\vert
&&\vert&\\
\bullet&&\bullet&\\
\vert&&&\\  
\bullet
& 
\stackrel{\line(1,0){15}}{}
\bullet 
\stackrel{\line(1,0){15}}{}
\bullet
 \stackrel{\line(1,0){15}}{}
 \bullet \stackrel{\line(1,0){15}}{}
 \bullet \stackrel{\line(1,0){15}}{}
&\bullet
& \stackrel{\line(1,0){15}}{}
 \bullet \stackrel{\line(1,0){15}}{}
 \bullet \\
&&\vert&\\
&&\bullet&\\
\end{array}
\end{equation}

Here the trivalent nodes correspond to the proper preimages of $l_1$ and $l_2$, and all the other nodes are the exceptional curves of the crepant resolution of one singularity of type $A_{11}$ and two each of type $A_2$ and $A_1$.

\medskip
We can now calculate the sublattice in the Picard group of $X$ generated by the above configuration of rational curves.
\begin{proposition}
The sublattice of $\Pic(X)$ generated by the rational curves from \eqref{graph}
is isomorphic to the unimodular lattice of rank $18$ given by $E_8(-1) \oplus E_8(-1) \oplus U$ which is 
embedded primitively into $\Pic(X)$.
Surface $X$ has a natural structure of elliptic fibration with a section and two fibers of type ${\rm II^*}$ in Kodaira's classification.
\end{proposition}

\begin{proof}
Specifically, the two copies of $E_8$ are generated by the two sides of the diagram, and the lattice $U$ perpendicular to them is generated by $S$ and $F$ given respectively by 
$$
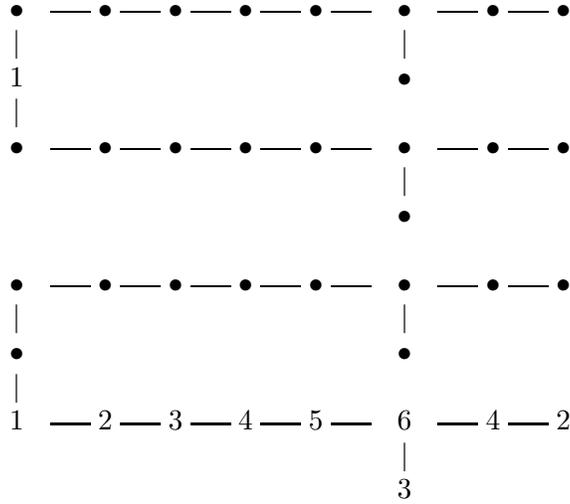
\begin{array} {cccl}
\bullet
& 
\stackrel{\line(1,0){15}}{}
\bullet 
\stackrel{\line(1,0){15}}{}
\bullet
 \stackrel{\line(1,0){15}}{}
 \bullet \stackrel{\line(1,0){15}}{}
 \bullet \stackrel{\line(1,0){15}}{}
&\bullet
& \stackrel{\line(1,0){15}}{}
 \bullet \stackrel{\line(1,0){15}}{}
 \bullet \\
\vert
&&\vert&\\
1&&\bullet&\\
\vert&&&\\  
\bullet
& 
\stackrel{\line(1,0){15}}{}
\bullet 
\stackrel{\line(1,0){15}}{}
\bullet
 \stackrel{\line(1,0){15}}{}
 \bullet \stackrel{\line(1,0){15}}{}
 \bullet \stackrel{\line(1,0){15}}{}
&\bullet
& \stackrel{\line(1,0){15}}{}
 \bullet \stackrel{\line(1,0){15}}{}
 \bullet \\
&&\vert&\\
&&\bullet&\\
\end{array}
$$
$$
\begin{array} {cccl}
\bullet
& 
\stackrel{\line(1,0){15}}{}
\bullet 
\stackrel{\line(1,0){15}}{}
\bullet
 \stackrel{\line(1,0){15}}{}
 \bullet \stackrel{\line(1,0){15}}{}
 \bullet \stackrel{\line(1,0){15}}{}
&\bullet
& \stackrel{\line(1,0){15}}{}
 \bullet \stackrel{\line(1,0){15}}{}
 \bullet \\
\vert
&&\vert&\\
\bullet&&\bullet&\\
\vert&&&\\  
1
& 
\stackrel{\line(1,0){15}}{}
2
\stackrel{\line(1,0){15}}{}
3
 \stackrel{\line(1,0){15}}{}
4\stackrel{\line(1,0){15}}{}
 5 \stackrel{\line(1,0){15}}{}
&6
& \stackrel{\line(1,0){15}}{}
 4\stackrel{\line(1,0){15}}{}
2 \\
&&\vert&\\
&&3&\\
\end{array}
$$
where $\bullet$ indicates coefficient $0$.
The embedding is primitive by unimodularity.

\medskip
We see that the second divisor above provides a fiber of elliptic fibration and the first provides the section.
Finally we observe that
$$
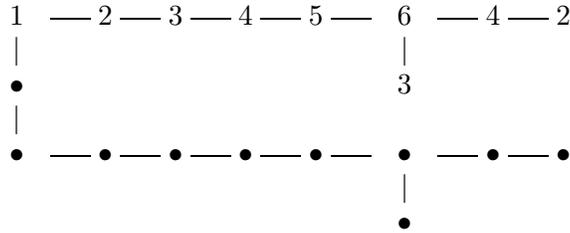
\begin{array} {cccl}
1
& 
\stackrel{\line(1,0){15}}{}
2
\stackrel{\line(1,0){15}}{}
3
 \stackrel{\line(1,0){15}}{}
4 \stackrel{\line(1,0){15}}{}
5 \stackrel{\line(1,0){15}}{}
&6
& \stackrel{\line(1,0){15}}{}
4 \stackrel{\line(1,0){15}}{}
2 \\
\vert
&&\vert&\\
\bullet&&3&\\
\vert&&&\\  
\bullet
& 
\stackrel{\line(1,0){15}}{}
\bullet 
\stackrel{\line(1,0){15}}{}
\bullet
 \stackrel{\line(1,0){15}}{}
 \bullet \stackrel{\line(1,0){15}}{}
 \bullet \stackrel{\line(1,0){15}}{}
&\bullet
& \stackrel{\line(1,0){15}}{}
 \bullet \stackrel{\line(1,0){15}}{}
 \bullet \\
&&\vert&\\
&&\bullet&\\
\end{array}
$$
provides another section of $F$. One way to see
it by noticing that the difference is the divisor of the monomial which corresponds to the lattice point $(1, 2, 3)$.
Alternatively, one can argue that since the family is two-dimensional, the very general member has Picard rank $18$, so it suffices to check that the difference intersects all of the curves from \eqref{graph} trivially. 
\end{proof}

\begin{remark}
It is clear that the fibers of this elliptic fibration are given by specifying the value 
of $z$ in \eqref{c-s} or \eqref{c-s-reduced}. Indeed, the rational function that corresponds to $(1,2,3)$ is precisely 
$z$, and this is not affected by the automorphisms that are used to reduce \eqref{c-s} to \eqref{c-s-reduced}.
\end{remark}

We will now describe a natural Morrison-Nikulin involution on $X$ (called Nikulin involution in \cite{Morrison}). This involution is supposed to have $8$ isolated fixed points so that the resolution of the quotient is again a K3 surface. By \cite{Morrison}
such involution should switch the two copies of $E_8$ and act trivially on its orthogonal complement $U$.
We recall the equation  \eqref{c-s-reduced} that $X=X(a,b)$ is the K3 compactification of $y^2+z+z^{-1}+x^3+ax+b=0$ and construct one such involution explicitly by 
$$
\mu:(x,y,z) \mapsto (x, -y, z^{-1})
$$
on $X_{singular}$ which then gives an involution on $X$.
We can also view this involution as a restriction to $X$ of an automorphism of the ambient toric variety.

\medskip
Note that this involution $\mu$ acts on the base of the elliptic fibration with two fixed points $z=\pm 1$. At these two
points, the action on the fiber is that of Kummer involution $(x,y)\mapsto (x,-y)$ for $0 = y^2 + x^3+ ax + b \pm 2$.
Thus, the involution has eight fixed points  at $y=0, z=\pm 1$ so the quotient $X/\mu$ has a crepant resolution which is itself a K3 surface. Let us denote this resolution by $\widetilde{X/\mu}$ and observe that it has the following tree of smooth rational curves in it. We get the following tree of rational curves on $\widetilde{X/\mu}$.
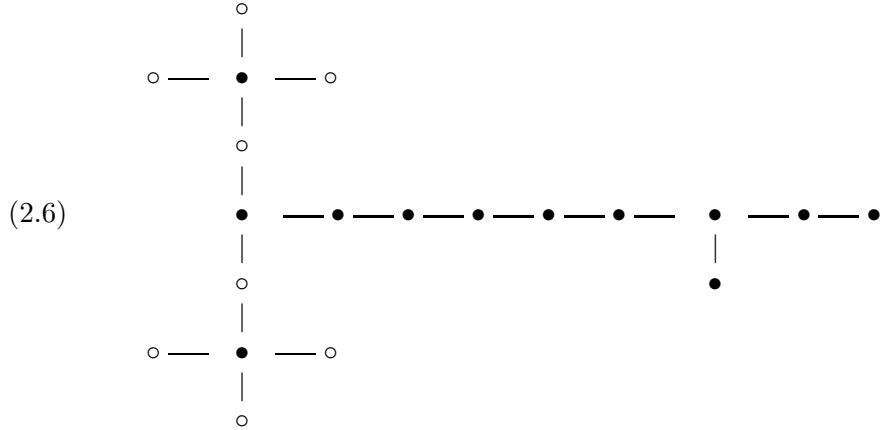
\begin{equation}\label{graph-mu}
\begin{array} {rcccl}
&\circ&&&\\
&\vert&&&\\
\circ \stackrel{\line(1,0){15}}{}&\bullet& \stackrel{\line(1,0){15}}{}\circ\hskip 130pt&&\\
&\vert&&&\\
 &\circ&&&\\
 &\vert&&&\\
&\bullet&
 \stackrel{\line(1,0){15}}{}
 \bullet \stackrel{\line(1,0){15}}{}
 \bullet \stackrel{\line(1,0){15}}{}
 \bullet \stackrel{\line(1,0){15}}{}
 \bullet \stackrel{\line(1,0){15}}{}
 \bullet \stackrel{\line(1,0){15}}{}
 &\bullet&
\stackrel{\line(1,0){15}}{}
 \bullet \stackrel{\line(1,0){15}}{} \bullet  
 \\
 &\vert&&\vert &\\
&\circ&&\bullet&\\
&\vert&&&\\
\circ \stackrel{\line(1,0){15}}{}&\bullet& \stackrel{\line(1,0){15}}{}\circ\hskip 130pt
&&
\\
&\vert&&&\\
&\circ&&&
\end{array}
\end{equation}
Geometrically the diagram \eqref{graph-mu} corresponds to an elliptic fibration with a section,  one fiber of type ${\rm II^*}$ and 
two fibers of type ${\rm I_0^*}$. The $\circ$ symbols indicate eight disjoint curves on $\widehat{X/\mu}$ 
which are the exceptional curves of $\widehat{X/\mu}\to {X/\mu}$. The 
double cover of $\widehat{X/\mu}$ ramified at these curves is the blowup of $X$ at the eight fixed points of $\mu$.

\begin{remark}\label{l3l4}
The classes of $l_3$ and $l_4$  and their images in $\widehat{X/\mu}$ can be computed, and we will need 
this description later.
The $(0,-1,-1)$ linear function has divisor $-l_3$ plus
$$
\begin{array} {cccl}
2
& 
\stackrel{\line(1,0){15}}{}
2
\stackrel{\line(1,0){15}}{}
2
 \stackrel{\line(1,0){15}}{}
2 \stackrel{\line(1,0){15}}{}
2 \stackrel{\line(1,0){15}}{}
&2
& \stackrel{\line(1,0){15}}{}
1 \stackrel{\line(1,0){15}}{}
0 \\
\vert
&&\vert&\\
2&&1&\\
\vert&&&\\  
2
& 
\stackrel{\line(1,0){15}}{}
2 
\stackrel{\line(1,0){15}}{}
2
 \stackrel{\line(1,0){15}}{}
2 \stackrel{\line(1,0){15}}{}
2 \stackrel{\line(1,0){15}}{}
&2
& \stackrel{\line(1,0){15}}{}
 1 \stackrel{\line(1,0){15}}{}
 0 \\
&&\vert&\\
&&1&\\
\end{array}
$$
so the divisor of $l_3$ is the above one. The fixed points of $\mu$ do not lie in $l_3$
but two of them do lie in the above tree of rational curves, namely at the zero section $S$ of 
the fibration.
Therefore the image on $\widehat{X/\mu}$ is given by the following divisor.
$$
\begin{array} {rcccl}
&0&&&\\
&\vert&&&\\
0 \stackrel{\line(1,0){15}}{}& 0 & \stackrel{\line(1,0){15}}{} 0 \hskip 130pt&&\\
&\vert&&&\\
 & 1 &&&\\
 &\vert&&&\\
& 2&
 \stackrel{\line(1,0){15}}{}
 2 \stackrel{\line(1,0){15}}{}
 2 \stackrel{\line(1,0){15}}{}
 2 \stackrel{\line(1,0){15}}{}
 2 \stackrel{\line(1,0){15}}{}
 2 \stackrel{\line(1,0){15}}{}
 &2&
\stackrel{\line(1,0){15}}{}
 1 \stackrel{\line(1,0){15}}{} 0
 \\
 &\vert&&\vert &\\
&1&&1&\\
&\vert&&&\\
0\stackrel{\line(1,0){15}}{}&0& \stackrel{\line(1,0){15}}{}0\hskip 130pt
&&
\\
&\vert&&&\\
&0&&&
\end{array}
$$
Similarly, the monomial $(0,-1,-2)$ gives $-l_4$ plus
$$
\begin{array} {cccl}
3
& 
\stackrel{\line(1,0){15}}{}
3
\stackrel{\line(1,0){15}}{}
3
 \stackrel{\line(1,0){15}}{}
3 \stackrel{\line(1,0){15}}{}
3 \stackrel{\line(1,0){15}}{}
&3
& \stackrel{\line(1,0){15}}{}
2 \stackrel{\line(1,0){15}}{}
1 \\
\vert
&&\vert&\\
3&&1&\\
\vert&&&\\  
3
& 
\stackrel{\line(1,0){15}}{}
3
\stackrel{\line(1,0){15}}{}
3
 \stackrel{\line(1,0){15}}{}
3 \stackrel{\line(1,0){15}}{}
3 \stackrel{\line(1,0){15}}{}
&3
& \stackrel{\line(1,0){15}}{}
 2 \stackrel{\line(1,0){15}}{}
 1 \\
&&\vert&\\
&&1&\\
\end{array}
$$
on $X$. The image of $l_4$ on $X/\mu$ is not Cartier, so it does not pull back to $\widehat{X/\mu}$. However,  $2l_4$ does and its class 
on $\widehat{X/\mu}$ is the following.
$$
\begin{array} {rcccl}
&0&&&\\
&\vert&&&\\
0 \stackrel{\line(1,0){15}}{}& 0 & \stackrel{\line(1,0){15}}{} 0 \hskip 130pt&&\\
&\vert&&&\\
 & 3 &&&\\
 &\vert&&&\\
& 6&
 \stackrel{\line(1,0){15}}{}
 6\stackrel{\line(1,0){15}}{}
 6 \stackrel{\line(1,0){15}}{}
 6 \stackrel{\line(1,0){15}}{}
 6 \stackrel{\line(1,0){15}}{}
 6 \stackrel{\line(1,0){15}}{}
 &6&
\stackrel{\line(1,0){15}}{}
 4 \stackrel{\line(1,0){15}}{} 
 2
 \\
 &\vert&&\vert &\\
&3&&2&\\
&\vert&&&\\
0\stackrel{\line(1,0){15}}{}&0& \stackrel{\line(1,0){15}}{}0\hskip 130pt
&&
\\
&\vert&&&\\
&0&&&
\end{array}
$$
\end{remark}

\medskip
The following result follows from the work of Morrison \cite{Morrison}, but we will be later able to 
establish it directly. 
\begin{theorem}\label{S-I}
For a generic choice of $a$ and $b$ there exist elliptic curves $E_1$ and $E_2$ such
that the above K3 surface $\widetilde{X/\mu}$ is isomorphic to the Kummer surface 
constructed from the product of elliptic curves $E_1\times E_2$.
\end{theorem}

\begin{proof}
The transcendental lattices of the abelian surface $E_1\times E_2$ and of $X$ are isomorphic, 
since they are both isomorphic to $U\oplus U$. Indeed, the embedding of $E_8(-1)\oplus E_8(-1)\oplus U$
into $H^2(X,\Zz)$ is unique up to isometry by \cite[Theorem 1.14.4]{Nikulin}, so the transcendental lattice is isomorphic to
$U\oplus U$.
This implies that $E_1\times E_2$ and $X$ form
Shioda-Inose pair, by \cite[Theorem 6.3]{Morrison}.
\end{proof}

\begin{remark}
While the above theorem establishes an isomorphism, it does not make it explicit. It also does not identify how the natural parameters on the moduli space 
of pairs of elliptic curves correspond to the parameters of $(a,b)$ of the family $X(a,b)$ in \eqref{c-s-reduced}. The next two sections make this isomorphism explicit, including the parameter identification.
\end{remark}

\section{Kummer surfaces associated to the product of two elliptic curves}\label{secKummer}

In this section we review the theory of Kummer surfaces as it associated to the product of two elliptic curves. Our main goal is to set up the notations that will be used in the subsequent sections.

\medskip
Let $E_1$ and $E_2$ be two elliptic curves.
Define the Kummer surface $Y $ as the minimal resolution of the quotient   $(E_1\times E_2)/\la \pm 1\ra$ of 
the product of these elliptic curves by the negation involution.


\medskip
There are four \emph{horizontal} smooth rational curves $F_{1,i}$ where $1\leq i\leq 4$ which correspond
to the proper preimage of $(E_1\times w)/\la \pm 1\ra \subset (E_1\times E_2)/\la \pm 1 \ra$ for a point $w\in E_2$
of order $2$. There are also four \emph{vertical} rational curves $F_{2,j}$ which are proper preimages of 
$(w\times E_2) /\la \pm 1\ra \subset (E_1\times E_2)/\la \pm 1 \ra$. These eight curves are disjoint from each other.

\medskip
There are also sixteen pairwise disjoint rational curves $G_{ij}$ which are the exceptional curves of $Y\to  (E_1\times E_2)/\la \pm 1 \ra$ that satisfy
$$
G_{ij} F_{1,k} = \delta_{i,k},~~G_{ij} F_{2,k} = \delta_{j,k}.
$$
where $\delta$ is the Kronecker symbol.

\medskip
Note that there is a natural map 
$$
Y\to \Pp^1\times \Pp^1 \cong (E_1/\la \pm 1\ra )\times (E_2 /\la\pm 1\ra)
$$
and the preimages of the $\Oo(0,1)$ and $\Oo(1,0)$ will be denoted by $F_1$ and $F_2$.
Each of these invertible sheaves defines an elliptic fibration structure on $Y$, with four fibers of type ${\rm I_0^*}$.
Specifically, for each $i$ we have the relations in Picard group of $Y$
$$
F_1  = 2F_{1,i} + \sum_{j=1}^4 G_{i,j},~~ F_2 = 2F_{2,i} + \sum_{j=1}^4 G_{j,i}.
$$

\begin{remark}
It can be shown that for very general choices of $E_1$ and $E_2$, the divisors $G_{ij},F_{1,i},F_{2,i}$ generate 
the Picard group of $Y$, but we will not need this statement.
\end{remark}

\begin{remark}\label{morecurves}
There are other smooth rational curves on $Y$. For example, consider the class $D=F_1+F_2 - G_{i_1,j_1}-G_{i_2,j_2}-G_{i_3,j_3}$.
Since $D^2 = -2$, and $DF_1=2>0$, it is an effective class. It is easy to see that 
if all $i_k$ are distinct and all $j_k$ are distinct then $h^0(Y,D)=1$ and it is given by a smooth rational curve which is a proper preimage
of a $(1,1)$ curve on $\Pp^1\times \Pp^1$ that passes through the three points that are the images of the type  $G$ divisors.
\end{remark}

\medskip 
We find it useful to denote the divisor classes on $Y$ by a pair of numbers and a matrix in order to better visualize them. For example,
$$
(3,2) ;~ \left(
\begin{array}{rrrr}
0 & 0 & -1 & 1\\
0& 0& 0& 0\\
4 & 0&0&0\\
0& 0& 0&0
\end{array}
\right)
$$
stands for $3F_1 + 2F_2 - G_{1,3}+ G_{1,4} + 4 G_{3,1}$. Some of the classes might require half-integers to be expressed in this manner. For example, $F_{1,3}$  and $F_{2,2}$ are represented by
$$
(1,0); ~
\left(
\begin{array}{rrrr}
0 & 0 & 0 & 0\\
0& 0& 0& 0\\
-\frac 12 &-\frac 12 &-\frac 12&-\frac 12\\
0& 0& 0&0
\end{array}
\right)
\mbox{~~and ~~}
(0,1); ~
\left(
\begin{array}{rrrr}
0 & -\frac 12 & 0 & 0\\
0& -\frac 12& 0& 0\\
0 &-\frac 12 & 0&0\\
0& -\frac 12& 0&0
\end{array}
\right)
$$
respectively. 
The self-intersection of the divisor given by  $(a,b); A$ is $4 a b - 2 \Tr (AA^T)$.

\section{Isomorphism}\label{seciso}
In this section we will construct  isomorphisms between the surfaces $Y$ and $\widetilde{X/\mu}$ from  previous two sections, which will make Theorem \ref{S-I} more
explicit. Note that such isomorphism is far from unique, since $Aut(Y)$ is infinite, with generators calculated in \cite{KK}.

\begin{remark}
The process of finding an explicit isomorphism started from looking for the generically $2:1$ map 
$\widetilde {X/\mu}\to (\Pp^1)^2$ that comes from taking the quotient by the  involution $(x,y,z)\mapsto (x,-y,z)$  and then contractions,
by identifying the branch locus of the map. However, the argument that will be presented in this paper goes from
the opposite direction. Thus, even though the construction may look like a magic trick of some sort, it is actually an outcome of a deliberate process, aided by Pari-GP.
\end{remark} 

Consider the divisor $D$ on $Y$ given by 
$$
(3,4);~
\left(
\begin{array}{rrrr}
-1 & -1 & -2 & 0\\
-2 & -2 &  0 & 0\\
0  &  0 & 0  & -3\\
0  &  0 & -1 & 0
\end{array}
\right)
$$
We have $D^2 = 48 - 2(6+8+9+1) = 0$, so $|D|$ gives an elliptic fibration, as long as it is base-point free.

\medskip
Observe that this divisor 
$$
D = 3F_1 +4F_2 - G_{1,1}-G_{1,2}-2G_{1,3}-2G_{2,1}-2G_{2,2}-3G_{3,4}-G_{4,3}
$$ 
intersects trivially the following nine irreducible divisors that are arranged in an affine $E_8$ diagram according 
to their intersections.
$$
\begin{array}{rcl}
 F_{1,1} \stackrel{\line(1,0){15}}{}
 G_{1,4} \stackrel{\line(1,0){15}}{}
 &
 F_{2,4}
 &
 \stackrel{\line(1,0){15}}{}
G_{2,4}
\stackrel{\line(1,0){15}}{}
F_{1,2}
\stackrel{\line(1,0){15}}{}
G_{2,3}
\stackrel{\line(1,0){15}}{}
F_{2,3}
\stackrel{\line(1,0){15}}{}
G_{3,3}
\\
&\vline&\\
&G_{4,4}& 
\end{array}
$$

Moreover, one has the identity in $\Pic(Y)$ 
\begin{equation}\label{e8fiber}
D = 2F_{1,1} + 4G_{1,4} + 3G_{4,4}+ 6F_{2,4} + 5G_{2,4}+4F_{1,2}+3G_{2,3}+2F_{2,3}+G_{3,3}.
\end{equation}
Since it can be seen that the restriction of $D$ to each of the above curves is trivial, an easy calculation shows that $h^0(D)=2$ and we see that $|D|$ gives
an elliptic fibration $Y\to \Pp^1$ with the above type ${\rm II^*}$ fiber. 

\medskip
We will also now establish two ${\rm I_0^*}$ fibers. 
We observe that $F_{2,1}$, $G_{3,1}$, $G_{4,1}$ intersect $D$ trivially, as does the curve
$C_1$ in class $\vert F_1+F_2 - G_{1,3} - G_{2,2} - G_{3,4}\vert $ considered in Remark \ref{morecurves}.
Note that $F_{2,1}$ intersects the other three curves at one point each, and these intersection points are distinct.
We denote by $C_2$ the difference
\begin{align*}
C_2 = &D - 2F_{2,1} - G_{3,1}-G_{4,1}-C_1 \\
=&2F_1+2F_2 -G_{1,2}-G_{1,3}-G_{2,1}-G_{2,2}-G_{4,3}  - 2G_{3,4}.
\end{align*}
Observe that the dimension count shows that there exists a $(2,2)$ curve on $\Pp^1\times \Pp^1$
with double vanishing at a point and vanishing at five  more points. Moreover, while this configuration of six
points is not generic, one can easily see that this curve is irreducible by looking at possible splittings of it.
We will also calculate the equation of this curve explicitly in the next section. Therefore, we get a ${\rm I_0^*}$ fiber of $Y\to \Pp^1$ with curves 
$$
\begin{array}{rcl}
&C_1&\\
&\vert &\\
C_2 \stackrel{\line(1,0){15}}{}& F_{2,1} &  \stackrel{\line(1,0){15}}{} G_{3,1}\\
&\vert &\\
&G_{4,1}&
\end{array}
$$
The other fiber consists of $F_{2,2}$, $G_{3,2}$, $G_{4,2}$ and two other rational curves $C_3$ and $C_4$ 
defined similarly to $C_1$ and $C_2$.

\medskip
Note that we also have an explicit section of the elliptic fibration which is given by $F_{1,3}$.
In fact, we can label the tree of rational curves in \eqref{graph-mu} by the corresponding curves
in $Y$ as follows. This labeling can be thought of as a key to our construction, both mathematically and visually.
\begin{equation}\label{graph-mu-labels}
\begin{array} {rcccl}
&C_1&&&\\
&\vert&&&\\
C_2 \stackrel{\line(1,0){10}}{}&F_{2,1}&\hskip -11pt \stackrel{\line(1,0){10}}{}G_{4,1}\hskip 150pt&&\\
&\vert&&&\\
 &G_{3,1}&&&\\
 &\vert&&&\\
&F_{1,3}&
\hskip -10pt
 \stackrel{\line(1,0){10}}{}
 G_{3,3} \stackrel{\line(1,0){10}}{}
 F_{2,3}\stackrel{\line(1,0){10}}{}
 G_{2,3} \stackrel{\line(1,0){10}}{}
 F_{1,2}\stackrel{\line(1,0){10}}{}
 G_{2,4}\stackrel{\line(1,0){10}}{}
 &F_{2,4}&
\stackrel{\line(1,0){10}}{}
 G_{1,4}\stackrel{\line(1,0){10}}{} F_{1,1}
 \\
 &\vert&&\vert &\\
&G_{3,2}&&G_{4,4}&\\
&\vert&&&\\
C_4\stackrel{\line(1,0){10}}{}&F_{2,2}& \hskip -11pt\stackrel{\line(1,0){10}}{}G_{4,2}\hskip 150pt
&&
\\
&\vert&&&\\
&C_3&&&
\end{array}
\end{equation}

\medskip
\begin{remark}
Now that we have established this  somewhat explicit
isomorphism between generic Kummer surfaces of $E_1\times E_2$ and 
the resolutions of the quotients $\widetilde {X/\mu}$, we would like to identify the natural involutions on them.
The Kummer surface has an involution with the fixed locus $\sqcup_iF_{1,i}\sqcup_iF_{2,i}$ 
such that the corresponding quotient is the blowup of $(\Pp^1)^2$ in $16$ points. The surface $\widetilde {X/\mu}$
has an involution that comes from the involution
$$
(x,y,z)\mapsto (x,-y,z)
$$
on $X$ that commutes with $\mu$. We claim that under our isomorphism these involutions coincide. Indeed, 
the composition of these involutions preserves the holomorphic two form on $Y\cong \widetilde {X/\mu}$. It also
has $F_{1,3}$ in its fixed locus. Since the fixed locus must be even dimensional, it means that the product of the involutions is the identity.
\end{remark}

\medskip
\begin{corollary}\label{coroF14}
The above remark implies that the remaining type $F$ curve $F_{1,4}$ is the three-section whose fibers consist of nonzero points of order two. Indeed, this three-section is fixed under $(x,y,z)\mapsto (x,-y,z)$.
\end{corollary}

\begin{remark}\label{auto}
One may switch the roles of $E_1$ and $E_2$ and consider the elliptic fibration $Y\to \Pp^1$ given by
the divisor
$$
(4,3);~
\left(
\begin{array}{rrrr}
-1 & -2 & 0 & 0\\
-1 & -2 &  0 & 0\\
-2  &  0 & 0  & -1\\
0  &  0 & -3 & 0
\end{array}
\right).
$$
It can be shown that this new divisor maps to $D$ under one of the generators of the automorphism group $Y$ described in \cite{KK}.
\end{remark}

\begin{remark}\label{YtoX}
We observe that $X$ can be obtained from $Y$ by taking a double cover ramified at the disjoint rational curves
$$
C_1,C_2,C_3,C_4,G_{3,1},G_{3,2},G_{4,1},G_{4,2}
$$
and then contracting the ramification locus.
\end{remark}

\begin{remark}
While the discussion of this section identifies the Picard lattice of the generic surfaces $Y$ and $\widehat{X/\mu}$, it does not provide us with an identification
of the parameters of the corresponding families. This will be accomplished in the next section by a direct calculation.
\end{remark}

\section{Explicit formulas}\label{secexp}
In this section we will make explicit the isomorphism from the previous section. We used Maple symbolic manipulation software, although in principle these formulas are simple enough to do by hand.

\medskip
We will explicitly compute two sections of $D$ on $Y$. We will assume that $Y$ is the resolution of the double
cover of $\Pp^1\times \Pp^1$ with homogeneous coordinates $(u_1:v_1)$ and $(u_2:v_2)$ ramified at
$$
\{u_2=0\} \cup\{v_2=0\} \cup \{u_2=v_2\} \cup \{u_2=\lambda_2 v_2\}
$$$$\cup
\{u_1=0\} \cup\{v_1=0\} \cup \{u_1=v_1\} \cup \{u_1=\lambda_1 v_1\}.
$$
The proper preimages of the above lines are $F_{1,1},\ldots, F_{1,4},F_{2,1},\ldots, F_{2,4}$
in this order.
Then we are looking for the homogeneous polynomials of bidegree $(4,3)$ with the vanishing conditions
prescribed by the matrix
$$\left(
\begin{array}{rrrr}
-1 & -1 & -2 & 0\\
-2 & -2 &  0 & 0\\
0  &  0 & 0  & -3\\
0  &  0 & -1 & 0
\end{array}
\right).
$$ 

\medskip
By \eqref{e8fiber} the $E_8$ fiber corresponds on $\Pp^1\times \Pp^1$ to the $(4,3)$ polynomial
$$H_\infty = (\lambda_2-1)(u_1-\lambda_1v_1)^3(u_1-v_1)u_2v_2^2.$$
It is scaled by $(\lambda_2-1)$ to simplify some further formulas.
   
\medskip
It is a bit harder to calculate the other fibers. Specifically, for the fiber $2F_{2,1}+G_{3,1}+G_{4,1}+C_1+C_2$
we need to know the equations of $C_1$ and $C_2$. The curve $C_1$ comes from a $(1,1)$ polynomial which 
vanishes on $(1:1; 0:1)$, $(1:0;1:0)$ and $(\lambda_1:1;1:1)$. Such polynomial is given by
$$
(\lambda_1 - 1)   v_1u_2  - u_1v_2 + v_1v_2.
$$
To find the polynomial of $C_2$ we need to find a $(2,2)$ polynomial which vanishes to the second order 
at 
$(\lambda_1:1;1:1)$
and vanishes at $(1:0;0:1)$, $(1:1;0:1)$, $(0:1;1:0)$, $(1:0;1:0)$, $(1:1;\lambda_2:1)$.
It is given by 
$$
\lambda_1(\lambda_1-1)v_1^2v_2^2 + \lambda_1(\lambda_1\lambda_2 -2\lambda_1+1) v_1^2u_2v_2
-\lambda_1(\lambda_1-1)u_1v_1v_2^2
$$$$ + (2\lambda_1^2-2\lambda_1\lambda_2-\lambda_1+1)u_1v_1u_2v_2
-(\lambda_1-1)^2u_1v_1u_2^2 +(\lambda_2-1)u_1^2u_2v_2.
$$
Overall, the $(4,3)$ polynomial is given by 
$$
H_{+} = u_1\,\Big((\lambda_1 - 1)   v_1u_2  - u_1v_2 + v_1v_2\Big)\cdot
$$
$$
\cdot\Big(\lambda_1(\lambda_1-1)v_1^2v_2^2 + \lambda_1(\lambda_1\lambda_2 -2\lambda_1+1) v_1^2u_2v_2
-\lambda_1(\lambda_1-1)u_1v_1v_2^2
$$$$ + (2\lambda_1^2-2\lambda_1\lambda_2-\lambda_1+1)u_1v_1u_2v_2
-(\lambda_1-1)^2u_1v_1u_2^2 +(\lambda_2-1)u_1^2u_2v_2
\Big).
$$

\medskip
For the other fiber we need to find the equations of 
$$C_3 \in
\vert F_1+F_2 - G_{1,3} - G_{2,1} - G_{3,4}\vert $$
and 
$$C_4\in 
\vert 
2F_1+2F_2 
 -G_{1,1}-G_{1,3}-G_{2,2}-G_{2,1}-G_{4,3}  - 2G_{3,4}
\vert.
$$
These polynomials are given by 
$$
(\lambda_1-1)u_1u_2 - \lambda_1 u_1v_2+ \lambda_1 v_1v_2
$$
and
$$
-\lambda_1^2(\lambda_2-1)v_1^2 u_2v_2 + (1-\lambda_1)u_1v_1v_2^2 +(-\lambda_1^2 + 2\lambda_1\lambda_2+\lambda_1-2)u_1v_1u_2v_2
$$
$$
+(\lambda_1-1)^2 u_1v_1 u_2^2 + (\lambda_1-1)u_1^2v_2^2 
+(-\lambda_1-\lambda_2+2)u_1^2u_2v_2
$$
respectively, 
which gives  
$$
H_- = v_1 \Big(
(\lambda_1-1)u_1u_2 - \lambda_1 u_1v_2+ \lambda_1 v_1v_2
\Big)\cdot
$$
$$\cdot
\Big(
-\lambda_1^2(\lambda_2-1)v_1^2 u_2v_2 + (1-\lambda_1)u_1v_1v_2^2 +(-\lambda_1^2 + 2\lambda_1\lambda_2+\lambda_1-2)u_1v_1u_2v_2
$$
$$
+(\lambda_1-1)^2 u_1v_1 u_2^2 + (\lambda_1-1)u_1^2v_2^2 
+(-\lambda_1-\lambda_2+2)u_1^2u_2v_2\Big).
$$
As expected, the three polynomials $H_\infty$, $H_+$ and $H_-$ are linearly dependent, specifically,
$$
H_\infty + H_+ + H_- = 0.
$$
Recall that the rational function $z+z^{-1}$ takes the value $\infty$, $2$ and $(-2)$ at the three respective fibers.
As such, we see that $z+z^{-1}$ is given by $2H_\infty^{-1}(H_+ - H_-)$, which fixes it, once we fix the order of ${\rm I_0^*}$ fibers.

\begin{remark}
The rational map $X-\to \widetilde{X/\mu} \cong Y$ is ramified along the eight divisors 
$$
C_1,C_2,C_3,C_4, G_{3,1}, G_{4,1},G_{3,2},G_{4,2}.
$$
This means that the field of rational functions of $X$ is the degree two extension of the field of the rational functions of the Kummer surface $Y$ obtained by attaching the square root of 
$$
H_+ H_-^{-1}.
$$
In the notations of \eqref{c-s-reduced}, this square root would be $\frac {z-1}{z+1}$.
\end{remark}

\medskip
Our next goal to explicitly calculate the embedding into a toric variety that allows one to write $Y$ as a quotient 
of hypersurface of \eqref{c-s-reduced}. In order to accomplish this, 
we will now consider the meaning of $x$ and $y$, based on Remark \ref{l3l4}.

\medskip
First observe that $x$ is a rational function which is the ratio of two sections of the divisor of $l_3$.
These are given by 
\begin{equation}\label{dec}
G_{3,1}+G_{3,2}+2F_{1,3}+2G_{3,3}+2F_{2,3}+2G_{2,3}+2F_{1,2}+2G_{2,4}+2F_{2,4}+G_{1,4}+G_{4,4}
\end{equation}
$$
=2F_1+2F_2 -G_{1,3}-G_{4,3}
 -G_{2,1}-G_{2,2}-2G_{3,4}.
$$
In addition to the divisor \eqref{dec} above that corresponds to 
$$
(u_1-v_1)(u_1-\lambda_1v_1)(u_2-v_2)v_2
$$
one has
$$
(\lambda_1^2 \lambda_2-2\lambda_1^2+\lambda_1) v_1^2 u_2v_2 
+(2\lambda_1-\lambda_1^2-1)u_1v_1 u_2^2 
+( 2\lambda_1^2-2\lambda_1\lambda_2-\lambda_1+1) u_1v_1u_2v_2
$$
$$
+(\lambda_1-\lambda_1^2) u_1v_1v_2^2
+(\lambda_2-1)u_1^2u_2v_2.
$$
We denote by $x_1$  the ratio
$$
\Big(
(\lambda_1^2 \lambda_2-2\lambda_1^2+\lambda_1) v_1^2 u_2v_2 
+(2\lambda_1-\lambda_1^2-1)u_1v_1 u_2^2 
+( 2\lambda_1^2-2\lambda_1\lambda_2-\lambda_1+1) u_1v_1u_2v_2
$$
$$
+(\lambda_1-\lambda_1^2) u_1v_1v_2^2
+(\lambda_2-1)u_1^2u_2v_2\Big)
(u_1-v_1)^{-1}(u_1-\lambda_1v_1)^{-1}(u_2-v_2)^{-1}v_2^{-1}.
$$
It will be equal to $x$ up to a change of coordinates.

\medskip
We see that $y^2$ makes sense as a rational function on $X/\mu$. In a generic fiber of $X/\mu\to \Pp^1$ 
it has zeros of order two  
at points of order two and pole of order six at zero. Note that the three-section of points of order two is identified with $F_{1,4}$ by Corollary \ref{coroF14}. On the resolution $\widetilde {X/\mu}$ it will have zeros of order one at $C_1,C_2,G_{4,1},G_{4,2},C_3,C_4$  and poles of order three at $G_{3,1},G_{3,2}$. Since we know that its divisor
intersects all curves at degree zero, we can figure out the contribution of the components
of the $E_8$ fiber. This gives the divisor of $y^2$ as
$$
2F_{1,4} - \begin{array} {rcccl}
&(-1)&&&\\
&\vert&&&\\
(-1)\stackrel{\line(1,0){12}}{}& 0 & \hskip -10pt \stackrel{\line(1,0){12}}{} (-1) \hskip 100pt&&\\
&\vert&&&\\
 & 3&&&\\
 &\vert&&&\\
& 6&
\hskip -10pt \stackrel{\line(1,0){12}}{}
 6\stackrel{\line(1,0){12}}{}
 6 \stackrel{\line(1,0){12}}{}
6\stackrel{\line(1,0){12}}{}
 6 \stackrel{\line(1,0){12}}{}
 6 \stackrel{\line(1,0){12}}{}
 &6&
\stackrel{\line(1,0){12}}{}
 4 \stackrel{\line(1,0){12}}{} 
 2
 \\
 &\vert&&\vert &\\
&3&&2&\\
&\vert&&&\\
(-1)\stackrel{\line(1,0){10}}{}&0&\hskip-10pt \stackrel{\line(1,0){10}}{}(-1)\hskip 100pt
&&
\\
&\vert&&&\\
&(-1)&&&
\end{array}
$$
which is consistent with the formula for the divisor of $2l_4$ from Remark \ref{l3l4} and Corollary \ref{coroF14}.
So we can write the rational function of $y^2$ in terms of $(u_1:v_1,u_2:v_2)$ 
as 
$$
y_1^2 = (u_2-\lambda_2 v_2) H_+ H_-  u_1^{-1} v_1^{-1} (u_2-v_2)^{-3} (u_1-v_1)^{-3}v_2^{-3} (u_1-\lambda_1 v_1)^{-3} u_2^{-1}
$$
up to a constant factor.

\medskip
A Maple calculation verifies that 
$$
x_1^3 + (\lambda_1\lambda_2 - 2\lambda_1 + \lambda_2 +1) x_1^2 -(\lambda_1\lambda_2 -\lambda_1+1)(\lambda_1-\lambda_2)x_1 -\frac 12(\lambda_1-1)(\lambda_2-1)\lambda_1\lambda_2 
$$
$$
+y_1^2 - \frac 14\lambda_1(\lambda_1-1)\lambda_2(\lambda_2-1)(z+z^{-1})= 0.
$$
By a linear change of variables this equation can be rewritten in the form \eqref{c-s-reduced} 
$$
y^2+z+z^{-1}+x^3+ax+b=0
$$
with 
$$
a=-\frac {2^{\frac 43}}3 \frac {(\lambda_1^2-\lambda_1+1)(\lambda_2^2-\lambda_2+1)}{\Big(\lambda_1(\lambda_1-1)
\lambda_2(\lambda_2-1)\Big)^{\frac 23}},
$$$$
b=-\frac 2{27}\frac {(\lambda_1+1)(\lambda_1-2)(2\lambda_1-1)(\lambda_2+1)(\lambda_2-2)(2\lambda_2-1)}{\lambda_1(\lambda_1-1)
\lambda_2(\lambda_2-1)}.
$$
By using the formula \cite{Silverman}
$$
j =256 \frac{(\lambda^2-\lambda+1)^3}{\lambda^2(\lambda-1)^2},
$$
we can rewrite the coefficients $a$ and $b$ 
in terms of the $J$-invariants $j_1$ and $j_2$ of the elliptic curves $E_1$ and $E_2$ as
$$
a=  -\frac{1} {48}j_1^{\frac 13}j_2^{\frac 13},\hskip 20pt
b=-\frac{ (j_1-1728)^{\frac 12}(j_2-1728)^{\frac 12}}{864}
$$
where different choices of the roots lead to isomorphic surfaces.

\medskip
We will state this as a main result of this paper.
\begin{theorem}\label{main}
For generic choices of elliptic curves $E_1$ and $E_2$ with $J$-invariants $j_1$ and $j_2$ the 
surface $X$ which is the compactification of
the solution space of
\begin{equation}\label{jinv}
y^2 + z+z^{-1} + x^3  -\frac{j_1^{\frac 13}j_2^{\frac 13} }{48}\,x  -\frac{ (j_1-1728)^{\frac 12}(j_2-1728)^{\frac 12}}{864} = 0
\end{equation}
and $E_1\times E_2$ form a Shioda-Inose pair. Specifically, the minimal resolution 
of the quotient of $X$ by $\mu:(x,y,z)\mapsto (x,-y,z^{-1})$ 
is isomorphic to the Kummer surface of $E_1\times E_2$.
\end{theorem}

\begin{remark}
Even though the statement of Theorem \ref{main} is symmetric with respect to the interchange of $E_1$ and $E_2$, the isomorphism provided by our construction depends on the choice of the ordering. A different choice gives an isomorphism that differs by an automorphism of the Kummer surface of $E_1\times E_2$, see Remark \ref{auto}. 
\end{remark}

\begin{remark}\label{n1}
We can observe that there is a map from a double cover of $Y$ to $X$ under for any pair of values of the $J$-invariant.
Indeed, all of the formulas of this section are applicable as long as $\lambda_i\not\in \{0,1\}$. The only cases where Shioda-Inose structure is not given by Theorem \ref{main} are the cases where the surface acquires additional singularities. These would only occur at $y=0$, $z=\pm 1$. So we need to exclude the loci where $x^3+ax+b\pm 2$ have a double root. It turns out that this happens exactly when $j_1=j_2$, i.e. $E_1\cong E_2$.
\end{remark}

\section{One-dimensional subfamilies}\label{secdim1}
Our initial interest in this problem was motivated by the study of mirrors of $19$-dimensional families of polarized K3 surfaces of generic Picard rank one, as in \cite{Dolgachev}. It has been shown in \cite{Dolgachev} that the mirrors of K3 surfaces with
$\la 2n\ra$-polarization are one-parameter families birational to the double covers of Kummer surfaces of $E_{\tau} \times E_{-\frac 1{n\tau}}$ as $\tau$ varies in the upper half plane (or in the moduli curve $X_0(n)_+$). These mirror families are marked with a lattice $E_8(-1)\oplus E_8(-1)\oplus U\oplus \la -2n\ra$.

\medskip
Indeed as we specialize to $E_{\tau} \times E_{-\frac 1{n\tau}}$, the equations \eqref{jinv} will yield
surfaces of Picard rank $19$. 

\medskip
\begin{theorem}\label{dim1}
For $n>1$ a very general $\tau$ in the upper half plane, the K3 surface which is the compactification of 
$$
y^2+z+z^{-1} + x^3 -\frac{j(\tau)^{\frac 13}j(-\frac 1{n\tau})^{\frac 13} }{48}\,x  -\frac{ (j(\tau)-1728)^{\frac 12}(j(-\frac 1{n\tau})-1728)^{\frac 12}}{864} = 0
$$
has Picard lattice isomorphic to  $E_8(-1)\oplus E_8(-1)\oplus U\oplus \la -2n\ra$.
\end{theorem}

\begin{proof}
This statement follows from the fact that $E_1\times E_2$ and $X$ form Shioda-Inose pair for generic $\tau$, since our 
explicit isomorphism still applies, and the calculation of the transcendental lattice in \cite{Dolgachev}. 

\medskip
The only subtlety is that in $n=1$ case the surface $X$ is singular, see Remark \ref{n1}. Then the correct Shioda-Inose partner is the resolution of the $A_1$ singular point of  
$X$, which yields the Picard group $E_8(-1)\oplus E_8(-1)\oplus U\oplus \la -2\ra$ as expected.
\end{proof}

We can see the Picard lattice of $X$ from Theorem \ref{dim1} more explicitly as follows.
The defining feature of $E_{\tau} \times E_{-\frac 1{n\tau}}$ is that the corresponding curves 
have $n:1$ isogenies to each other. To describe one of them we will view $E_\tau$ and $E_{-\frac 1{n\tau}}$ as quotients
of $\Cc$ by $\Zz + \Zz\tau$ and $\Zz+\Zz(-\frac 1{n\tau})$ respectively and observe that 
there is a group homomorphism 
$\rho: E_{\tau} \to E_{-\frac 1{n\tau}}$ given by 
$$
z\mod (\Zz + \Zz\tau)\mapsto \frac z\tau \mod (\Zz+\Zz(-\frac 1{n\tau})). 
$$
This group homomorphism gives an elliptic fibration 
$$\rho: E_{\tau} \times E_{-\frac 1{n\tau}} \to E_{-\frac 1{n\tau}},~~
(z_1,z_2)\mapsto z_2+\rho(z_1)$$ 
which then gives rise to an elliptic fibration
$$
\rho_Y:Y \to \Pp^1 \cong E_{-\frac 1{n\tau}}/(-).
$$
The general fiber of $\rho_Y$ is a divisor $R_Y$ on $Y$ which intersects $G_{i,j}$ trivially and satisfies
$$
R_Y F_1 = 2n,~R_YF_2 = 2.
$$
We observe that $(R_Y- n F_2 - F_1) $ is orthogonal to all of the $F_{i,j}$ and $G_{i,j}$ divisors on $Y$ and
$$
(R_Y- n F_2 - F_1)^2 = -4n -4n + 4n =-4n.
$$

\medskip
Recall that by Remark \ref{YtoX} the K3 surface $X$ is obtained from a double cover of $Y$ at eight divisors 
that are linear combinations of $F_{i,j}$ and $G_{i,j}$ and then contraction of eight $(-1)$ curves. We can thus
view $\Pic(X)$ as an orthogonal complement in the double cover to the $8$ exceptional divisors.
The pullback of $(R_Y-nF_2-F_1)$ to the double cover and then orthogonal projection to an element $R_X\in \Pic(X)$ 
will be orthogonal to the rank $18$ sublattice $E_8(-1)\oplus E_8(-1)\oplus U$ described in Section \ref{toricsection} and will have self-intersection 
\begin{equation}\label{8n}
R_X^2 = -8n.
\end{equation} 
Note however that $R_X$ and the aforementioned sublattice do not generate the whole Picard group of $X$. Namely, there are four reducible fibers of $\rho_Y$ which correspond to points of order two on $E_{-\frac 1{n\tau}}$. They give rise to elements of Picard group of $Y$
which are half of $R$ modulo the lattice generated by $F$ and $G$ divisors. By pullback and projection, we get 
elements of $\Pic(X)$.
Since $E_8(-1)\oplus E_8(-1)\oplus U$ is unimodular, and Picard of $X$ is generically $19$, the lattice $\Pic(X)$ is 
equal to $E_8(-1)\oplus E_8(-1)\oplus U$ plus its orthogonal complement which is generated by a single element. We know
that it will have self intersection $(-2n)$ by \eqref{8n}.

\medskip
\begin{remark}
It would be interesting to connect  our Theorem \ref{dim1} to the work of Dolgachev \cite{Dolgachev} for small values of $n$, using the known formulas for classical modular polynomials \cite{Silverman} which are polynomial equations that vanish on $(j(\tau),j(-\frac 1{n\tau}))$. However, these classical modular polynomials have rapidly growing complexity. The resulting K3 surfaces has very rich geometric structure with multiple elliptic fibrations, which would make explicit comparisons rather complicated.
\end{remark}

\begin{remark}
We will also briefly comment that Theorem \ref{main} can be used in a uniform way to provide many example of K3 surfaces of Picard rank $20$ by looking at pairs of isogeneous curves with complex multiplication, although we are unaware of any potential applications.
\end{remark}


\begin{thebibliography}{99999}
\bibitem[B]{Bat.dual}
V. Batyrev, \emph{Dual polyhedra and mirror symmetry for Calabi-Yau hypersurfaces in toric varieties}, J. Algebraic Geom. 3 (1994), no. 3, 493-535. 

\bibitem[CD]{DoranClingher}
A. Clingher, C. Doran, \emph{Modular invariants for lattice polarized K3 surfaces.} Michigan Math. J. 55 (2007), no. 2, 355-393.

\bibitem[D]{Dolgachev}
I. Dolgachev, \emph{Mirror symmetry for lattice polarized K3 surfaces},
Algebraic geometry, 4. J. Math. Sci. 81 (1996), no. 3, 2599-2630.

\bibitem[GNS]{GNS}
V. Gritsenko, K. Hulek, G. Sankaran, \emph{The Kodaira dimension of the moduli of K3 surfaces}, Invent. Math. 169 (2007), no. 3, 519-567.

\bibitem[H]{Huybrechts}
D. Huybrechts, \emph{Lectures on K3 Surfaces.} Cambridge Studies in Advanced Mathematics, 2016.

\bibitem[KK]{KK}
J. Keum, S. Kond\=o, \emph{
The automorphism groups of Kummer surfaces associated with the product of two elliptic curves},
Transactions of the AMS 353 (2001), no. 4, p. 1469-1487.

\bibitem[Kh]{Khovanskii}
A. Khovanskii, \emph{Newton polyhedra and the genus of complete intersections}, Funct. Anal. i ego pril. English translation: Functional Anal. Appl., 12 (1978), 38-46.

\bibitem[KS]{KuwataShioda} M. Kuwata, T. Shioda, \emph{Elliptic parameters and defining equations for elliptic fibrations on a Kummer surface.} Algebraic geometry in East Asia - Hanoi 2005, 177-215, Adv. Stud. Pure Math., 50, Math. Soc. Japan, Tokyo, 2008. 

\bibitem[M]{Morrison}
D. Morrison, \emph{On K3 surfaces with large Picard number},
Invent. Math. 75 (1984), no. 1, p. 105-121.

\bibitem[N]{Nikulin}
Nikulin, \emph{Integral symmetric bilinear forms and some of their applications}, Izv. Akad. Nauk SSSR,
43, p. 111-177 (1979), Math. USSR Izvestija 14, p. 103-167 (1980).

\bibitem[R]{Rohsiepe}
F. Rohsiepe, \emph{Lattice polarized toric K3 surfaces}, preprint arXiv:hep-th/0409290.

\bibitem[Sh]{ShiodaSandwich}
T. Shioda, \emph{Kummer sandwich theorem of certain elliptic K3 surfaces.} Proc. Japan Acad. Ser. A Math. Sci. 82 (2006), no. 8, 137-140. 

\bibitem[Si]{Silverman}
J. Silverman,  \emph{Advanced topics in the arithmetic of elliptic curves.}
Graduate Texts in Mathematics, 151. Springer-Verlag, New York, 1994. xiv+525 pp. ISBN: 0-387-94328-5.


\end{thebibliography}
\end{document}